\documentclass[12pt,a4paper]{article}

\usepackage{amsfonts}
\usepackage{amsmath}
\usepackage{amsthm}
\usepackage{graphicx}
\usepackage{xcolor, enumerate}
\usepackage{fullpage}

\newcommand{\PP}{\mathbb P}
\newcommand{\Z}{\mathbb Z}
\newcommand{\N}{\mathbb N}
\newcommand{\R}{\mathbb R}
\newcommand{\T}{\mathcal T}
\newcommand{\G}{\mathcal G}
\newcommand{\Pcal}{\mathcal P}
\newcommand{\eps}{\epsilon}

\DeclareMathOperator{\ch}{Conv}
\DeclareMathOperator{\sch}{SymConv}

\newtheorem{thm}{Theorem}
\newtheorem{lemma}[thm]{Lemma}
\newtheorem{cor}[thm]{Corollary}
\newtheorem*{defn}{Definition}
\newtheorem{claim}{Claim}
\newtheorem{q}[thm]{Question}

\begin{document}

\title{Asymmetry of $2$-step Transit Probabilities in $2$-Coloured Regular Graphs} 
\author{Ron Gray\thanks{School of Mathematical Sciences, Queen Mary University of London, Mile End Road, London E1 4NS, United Kingdom.}
\and
J. Robert Johnson\thanks{School of Mathematical Sciences, Queen Mary University of London, London E1 4NS, UK.  E-mail: r.johnson@qmul.ac.uk.}}
\maketitle

\begin{abstract}
Suppose that the vertices of a regular graph are coloured red and blue with an equal number of each (we call this a balanced colouring). Since the graph is undirected, the number of edges from a red vertex to a blue vertex is clearly the same as the number of edges from a blue vertex to a red vertex. However, if instead of edges we count walks of length $2$ which do not stay within their starting colour class, then this symmetry disappears. Our aim in this paper is to investigate how extreme this asymmetry can be.

Our main question is: Given a $d$-regular graph, for which pairs $(x,y)\in[0,1]^2$ is there a balanced colouring for which the probability that a random walk starting from a red vertex stays within the red class for at least $2$ steps is $x$, and the corresponding probability for blue is $y$? 

Our most general result is that for any $d$-regular graph, these pairs lie within the convex hull of the $2d$ points $\left\{\left(\frac{l}{d},\frac{l^2}{d^2}\right),\left(\frac{l^2}{d^2},\frac{l}{d}\right) :0\leq l\leq d\right\}$.

Our main focus is the torus for which we prove both sharper bounds and existence results via constructions. In particular, for the $2$-dimensional torus we show that asymptotically the region in which these pairs of probabilities can lie is exactly the convex hull of:
\[
\left\{\left(0,0\right),\left(\frac{1}{2},\frac{1}{4}\right),\left(\frac{3}{4},\frac{9}{16}\right),\left(\frac{1}{4},\frac{1}{2}\right),\left(\frac{9}{16},\frac{3}{4}\right),\left(1,1\right)\right\}
\]
\end{abstract}

\section{Introduction}

The aim of this paper is to investigate some extremal properties of a notion of edge boundary coming from random walks on a graph.

Let $G$ be a $d$-regular graph on $n$ vertices. The edge-boundary $\partial(S)$ of a set of vertices $S\subseteq V(G)$ is defined by 
\[
\partial(S)=\{e\in E(G): |e\cap S|=1\}
\]
(that is the set of edges which go between $S$ and $S^c$). 

The problem of minimising $\partial(S)$ for a given $|S|$ is the edge isoperimetric inequality problem in extremal combinatorics which has been studied for various graphs. Most notably the edge isoperimetric inequality was proved in the hypercube by Harper \cite{harper}, Lindsey \cite{lindsey}, Bernstein \cite{bernstein} and Hart \cite{hart} and in the grid by Bollob\'as and Leader \cite{bollobas-leader}. See also \cite{bollobas-comb} for extremal combinatorics background.

Suppose that we colour the vertices of $G$ with red and blue. In other words, we partition $V(G)$ into two colour classes $R$ and $B$. Clearly $\partial(R)=\partial(B)$ as each of these counts $E(R,B)$: the set of edges with one end in each colour class. However, as we shall see, there is a natural extension of edge boundary for which the red to blue boundary may not be the same as the blue to red boundary.

To define this, we first relate the edge boundary to probabilities in random walks on $G$. A random walk on $G$ moves from vertex to vertex, at each discrete time step choosing uniformly from all the neighbours of its current location. For $S\subseteq V(G)$, let $(W^S_i)_{i=0}^{\infty}$ be the random walk on $G$ in which $W^S_0$ is a uniformly random vertex in $S$. We will be interested in the probability that the random walk starting in a given colour class stays within this colour class for a fixed number of steps. We define these probabilities as:
\[
P_t(R)=\PP(W^R_1,W^R_2,\ldots, W^R_t\in R), \quad P_t(B)=\PP(W^B_1,W^B_2,\ldots, W^B_t\in B).
\]

It is easy to see that the probability we stay inside the same colour class for one step is determined by the edge boundary of the colour class. 
\[
P_1(R)=1-\frac{\partial(R)}{d|R|}, \quad P_1(B)=1-\frac{\partial(B)}{d|B|}
\]
In particular, if we have a balanced colouring (that is $|R|=|B|$) then $P_1(R)=P_1(B)$ because $\partial(R)=\partial(B)$. However, for $t\geq 2$, it is not necessarily the case that $P_t(R)$ and $P_t(B)$ are the same. 

Like the edge boundary, the pair $(P_t(R),P_t(B))$ provides some quantitative measure of how interspersed the sets $R$ and $B$ are. But unlike the edge boundary, this quantity is not symmetric in the two coordinates (at least when $t\geq 2$). Our main aim in this paper is to investigate this asymmetry at the smallest $t$ for which it occurs: the $2$-step walks. 

We write $P_t(R,B)$ for the pair $(P_t(R),P_t(B))\in[0,1]^2$. Our main, and very general, question is:  

\begin{q}
Given a particular $d$-regular graph $G$ with $|V(G)|$ even and a balanced colouring $(R,B)$, what possible values in $[0,1]^2$ can the pair $P_2(R,B)$ take?
\end{q}

Our first main result is that for any $d$-regular graph, the pair $P_2(R,B)$ must lie in the region $D_d\subseteq[0,1]^2$ defined as the convex hull of the following $2d$ extreme points: 
\[
\left\{\left(\frac{l}{d},\frac{l^2}{d^2}\right):0 \leq l \leq d\right\} \cup \left\{\left(\frac{l^2}{d^2},\frac{l}{d}\right):0 \leq l \leq d\right\}.
\]

As a simple example, we briefly describe the behaviour for some $2$-regular graphs. For the cycle $C_n$, the set above is essentially best possible, but for a disjoint union of $4$-cycles the true answer is substantially smaller.

The heart of the paper concerns the discrete torus $[k]^m$ and in particular, bounding the possible region for $P_2(R,B)$ asymptotically for large $k$. This is a $2m$-regular graph so by our first result this region is contained in $D_{2m}$.

For all $m$ we show that the extreme points of the region $D_{2m}$ corresponding to $l=0,m,2m-2,2m-1,m$ can be achieved but that those corresponding to $1\leq l\leq m-1$ cannot. The constructions showing these points can be achieved are based on independent sets $S\subseteq\Z^m$ with the property that every element of $S^c$ has exactly $r$ neighbours in $S$. We call these independent exact $r$-covers. For the $2$-dimensional torus it is possible to give a more intricate argument proving an exact result (asymptotically for large $k$). We show that for $m=2$ the possible region for $P_2(R,B)$ is significantly smaller than $D_4$; it is the convex hull of the following set of points
\[
\left\{\left(0,0\right),\left(\frac{1}{2},\frac{1}{4}\right),\left(\frac{3}{4},\frac{9}{16}\right),\left(1,1\right),\left(\frac{1}{4},\frac{1}{2}\right),\left(\frac{9}{16},\frac{3}{4}\right)\right\}.
\]
We finish with some questions and open problems.

We remark that another measure of boundary related to random walks is the expectation of the escape time (the smallest time at which the random walk $(W^S_i)_{i=0}^{\infty}$ reaches a vertex outside $S$). The natural extremal question here is for a given $|S|$, which $S$ maximises the expected escape time? In the setting of the hypercube, Samorodnitsky \cite{samorodnitsky} showed that the extremal sets are subcubes (exactly those sets which minimise edge boundary).

Finally, we note that while our motivation is purely mathematical, this notion of using random walks to measure spatial dispersion has been used to quantify socioeconomic segregation in various contexts. In these applications, the fact that it may be easier for a random walk to move from $R$ to $B$ than from $B$ to $R$ can be socially significant. For examples of work in this direction see \cite{BV,enzo2,enzo1}.

\section{Notation and Preliminaries}

Throughout, $G=(V,E)$ will be a $d$-regular graph with $n$ vertices and $n$ will be even. We will consider balanced colourings of $V$, that is partitions $V=R\cup B$ with $|R|=|B|=\frac{n}{2}$. For a vertex $v\in R$ we write $d_R$ for the degree of $v$ in the subgraph induced by $R$:
\[
d_R(v)=|\{u\in R : (u,v)\in E\}|.
\]
We write $R_i$ for the set of vertices in $R$ with $d_R(v)=i$:
\[
R_i=\{v\in R : d_R(v)=i\}
\]
and $r_i=|R_i|$.

The internal degree sum $I(R)$ of $R$ (or twice the number of internal edges in $R$) is
\[
I(R)=\sum_{v\in R} d_R(v)=\sum_{i=1}^d ir_i
\]
We similarly define $B_i, b_i$ and $I(B)$ for the colour class $B$. As our colouring is balanced $I(B)=d|B|-\partial(B)=d|R|-\partial(R)=I(R)$.

The $r_i$ also determine the $2$-step probabilities $P_2(R)$ as the following Lemma shows.

\begin{lemma}\label{squares}
If $G$ is an $n$-vertex, $d$-regular graph then
\[
P_2(R)=\frac{2}{nd^2}\sum_{v\in R} d^2_R(v)=\frac{2}{nd^2}\sum_{i=1}^d i^2r_i
\]
\end{lemma}

\begin{proof}
The number of $2$-step walks consisting entirely of vertices in $R$ and with middle vertex $u$ is $d_R(u)^2$ (we need to choose the first and last vertex in the walk independently from among all neighbours of $u$ in $R$). Hence the total number of $2$-step walks consisting entirely of vertices in $R$ is
\[
\sum_{u\in R} d_R(u)^2=\sum_{i=1}^d i^2r_i.
\]
Since the total number of $2$-step walks in $G$ starting from a vertex in R is $|R|d^2$ we have
\[
P_2(R)=\frac{\sum_{i=1}^d i^2r_i}{|R|d^2}=\frac{2}{nd^2}\sum_{i=1}^d i^2r_i.
\]
\end{proof}

If $S=\{s_1,\ldots,s_m\}$ is a finite set of points in $\R^2$ we write $\ch(S)$ for the convex hull of $S$. If $P_2(R,B)=(x,y)$ then swapping the colours gives a colouring with $P_2(R,B)=(y,x)$. This means that our set of possible values for $P_2(R,B)$ is symmetric in the coordinates and it will be useful to define a symmetrised version of the convex hull to describe it. For $S\subseteq\R^2$ let $\overline{S}=\{(y,x) : (x,y)\in S\}$. We write $\sch(S)$ for the set $\ch(S\cup\overline{S})$. 

\section{Bounding the possible values of $P_2(R,B)$}

Our first main result determines a region in $[0,1]^2$ within which $P_2(R,B)$ lies for any $d$-regular graph.

\begin{thm}\label{main:thm}
Let $G$ be a $d$-regular graph with $|V(G)|$ even and let $D_d \subset [0,1]^2$ be the set $\sch\left(\left\{\left(\frac{l}{d},\frac{l^2}{d^2}\right):0 \leq l \leq d\right\}\right)$. We have $P_2(R,B)\in D_d$ for all balanced colourings $(R,B)$ of $G$. 
\end{thm}

\begin{proof}
Define $\alpha$ so that $\sum_{v\in R} d_R(v)=\alpha\frac{dn}{2}$ (or equivalently and more explicity, set $\alpha=1-\frac{2\partial(R)}{dn}$). Clearly $0\leq\alpha\leq 1$.

By convexity, $\sum_{v\in R} d^2_R(v)$ is maximised when the $d_R(v)$ are all $0$ or $d$. We then must have $\frac{\alpha n}{2}$ vertices with $d_R(v)=d$ and $\frac{\left(1-\alpha\right)n}{2}$ vertices with $d_R(v)=0$ and so by Lemma \ref{squares}
\[
P_2(R)=\frac{2}{nd^2}\sum_{i=1}^d i^2r_i\leq\frac{2}{nd^2}d^2\frac{\alpha n}{2}=\alpha.
\]
Again, by convexity, $\sum_{v\in R} d^2_R(v)$ is minimised when the $d_R(v)$ are as equal as possible. So certainly, we cannot achieve a smaller value than by taking $d_R(v)=\alpha d$ for all $v$. Hence,
\[
P_2(R)=\frac{2}{nd^2}\sum_{i=1}^d i^2r_i\geq\frac{2}{nd^2}d^2\alpha^2 d^2 \frac{n}{2}=\alpha^2.
\]
Now crucially we have that $\sum_{v\in B} d_B(v)=\sum_{v\in R} d_R(v)=\alpha\frac{dn}{2}$. So, the same bounds (with the same value of $\alpha$) are valid for the other colour class $B$. 

It follows that, if $P_2(R)=x$ then $\alpha\geq x$ and $P_2(B)\geq\alpha^2\geq x^2=P_2(R)^2$. Similarly, if $P_2(B)=y$ then $\alpha\geq y$ and $P_2(R)\geq\alpha^2\geq y^2=P_2(B)^2$. This gives the weaker result that
\[
P_2(R,B)\in\{(x,y) : y\geq x^2, x\geq y^2\}.
\]
In other words, the possible pairs $P_2(R,B)$ all lie in the region bounded by the curves $y=x^2$ and $x=y^2$.

In this bound we did not make use of the fact that the $d_R(v)$ are all integers. Incorporating this in the argument minimising $P_2(R)$ gives a stronger bound. If $l\leq\alpha d<l+1$, the minimum of $P_2(R)$ is attained when the $d_R(v)$ are all $l$ or $l+1$. More precisely, we have $\frac{n}{2}\left(l+1-\alpha d\right)$ of the vertices having $d_R(v)=l$ and $\frac{n}{2}\left(\alpha d-l\right)$ of the vertices having $d_R(v)=l+1$. This gives a better lower bound for $P_2(R)$ as follows:
\begin{align*}
P_2(R)&=\frac{2}{nd^2}\sum_{i=1}^d i^2r_i\\
&\geq\frac{2}{nd^2}\left(l^2\frac{n}{2}\left(l+1-\alpha d\right)+(l+1)^2\frac{n}{2}\left(\alpha d-l\right)\right)\\
&=\frac{1}{d^2}\left(\alpha d(2l+1)-l(l+1)\right).
\end{align*}

For any fixed $l$ this is a linear function in $\alpha$. Moreover, if $\alpha=\frac{l}{d}$ or $\alpha=\frac{l+1}{d}$ then we recover the bound $P_2(R)\geq\alpha^2$. 

Using the same argument as before to constrain the possible values for the pair $P_2(R,B)$ we get that these must lie in the convex region defined by the extreme points 
\[
\left\{\left(\frac{l}{d},\frac{l^2}{d^2}\right):0 \leq l \leq d\right\} \cup \left\{\left(\frac{l^2}{d^2},\frac{l}{d}\right):0 \leq l \leq d\right\},
\]
as required.
\end{proof}

\section{Simple $2$-regular examples}

In this section, we analyse the behaviour of $P_2(R,B)$ for some $2$-regular graphs as simple illustrative examples. Note that the set $D_2$ has extreme points $\{(0,0),(\frac{1}{2},\frac{1}{4}),(\frac{1}{4},\frac{1}{2}),(1,1)\}$. 

First consider the $n$-cycle, $C_n$. Suppose that $G=C_n$ with $V=\{1,\ldots,n\}$ and $E=\{(i,i+1) : 1\leq i\leq n-1\}\cup\{(n,1)\}$. We can construct colourings of this graph which get close to each of the four extreme points of $D_2$. To simplify the description we assume that $n$ is a multiple of $4$.
\begin{itemize}
\item $R=\{1,\ldots,\frac{n}{2}\}, B=\{\frac{n}{2}+1,\ldots,n\}$, $P_2(R,B)=\left(1-\frac{3}{n},1-\frac{3}{n}\right)$
\item $R=\{2,4,6,\ldots,n\}, B=\{1,3,5,\ldots n-1\}$, $P_2(R,B)=\left(0,0\right)$
\item $R=\{x : 1\leq x\leq \frac{3n}{4}, x\not\equiv 0 \mod 3\}, B=V\setminus R$, $P_2(R,B)=\left(\frac{1}{4},\frac{1}{2}+O\left(\frac{1}{n}\right)\right)$
\item $B=\{x : 1\leq x\leq \frac{3n}{4}, x\not\equiv 0 \mod 3\}, R=V\setminus B$, $P_2(R,B)=\left(\frac{1}{2}+O\left(\frac{1}{n}\right),\frac{1}{4}\right)$
\end{itemize}

Note, in the asymmetric third case each red vertex has exactly one red neighbour (so $P_2(R)=\left(\frac{1}{2}\right)^2$). With a few exceptions (vertices $\frac{3n}{4},\frac{3n}{4}+1,n-1,n$), the blue vertices come in two types:
\begin{itemize}
\item those in $\{1,\ldots,\frac{3n}{4}-1\}$ have no blue neighbours (so contribute $0$ to $P_2(B)$); 
\item those in $\{\frac{3n}{4}+2,\ldots,n-1\}$ are at distance greater than $2$ from the nearest red vertex (so contribute $1$ to $P_2(B)$).
\end{itemize}
Since there are equal numbers of these two types $P_2(B)=\frac{1}{2}+O(\frac{1}{n})$.

Suppose we have a colouring $c_1$ of $C_n$ with $P_2(R,B)=(x_1,y_1)$ and another colouring $c_2$ of $C_n$ with $P_2(R,B)=(x_2,y_2)$. We will colour $C_{2n}$ using $c_1$ on the first half and $c_2$ on the second half. That is our colouring $c$ is: 
\[
c(x)=\begin{cases}
c_1(x) & \text{ if $1\leq x\leq n$;}\\
c_2(x-n) & \text{ if $n+1\leq x\leq 2n$.}
\end{cases}
\]
Each vertex apart from $1,2,n-1,n,n+1,n+2,2n-1,2n$ (that is those within distance $2$ of the `join' between the colourings) behaves exactly as it does in colouring $c_1$ or colouring $c_2$. Hence, this new colouring satisfies 
\[
P_2(R,B)=\left(\frac{1}{2}x_1+\frac{1}{2}x_2+O\left(\frac{1}{n}\right),\frac{1}{2}y_1+\frac{1}{2}y_2+O\left(\frac{1}{n}\right)\right).
\]
Similarly, we can join together $a$ copies of colouring $c_1$ and $b$ copies of colouring $c_2$ to obtain a colouring of $C_{(a+b)n}$ with
\[
P_2(R,B)=\left(\frac{a}{a+b}x_1+\frac{b}{a+b}x_2+O\left(\frac{1}{n}\right),\frac{a}{a+b}y_1+\frac{b}{a+b}y_2+O\left(\frac{1}{n}\right)\right).
\]
Taking $n$ large enough, this operation produces colourings in which $P_2(R,B)$ can be arbitrarily close to any point in $D_2$.

We see from this example that constructions matching the extreme points of $D_d$ are key. Also it may only be possible to achieve these up to $O(\frac{1}{n})$ so we will mainly be interested in asymptotic results.

Before we leave the $d=2$ case, note that this kind of construction will not be possible for all $2$-regular graphs. For instance, if $G$ is a disjoint union of copies of $C_4$ then it is not too difficult to see that the extreme point $(\frac{1}{2},\frac{1}{4})$ cannot be achieved. In fact, for this graph the set of possible pairs for $P_2(R,B)$ is essentially $\sch\left(\{(0,0),(\frac{1}{2},\frac{2}{3}),(1,1)\}\right)$.

\section{Torus}

As we noted earlier we will mainly be interested in asymptotic results. The next definition formalises this.

\begin{defn} 
Let $\G=(G_k)_{k=1}^{\infty}$ be a sequence of $d$-regular graphs with $|V(G_k)|$ even for all $k$ and tending to $\infty$ with $k$. Define $\Pcal(\G)\subseteq \mathbb{R}^2$ by $(x,y)\in\Pcal(\G)$ if and only if for every $\epsilon>0$ there exists some $k$ and a balanced colouring of $G_k$ with $|P_2(R)-x|<\epsilon$ and $|P_2(B)-y|<\epsilon$.
\end{defn}

For fixed $m\geq 2$, let $T^m_k$ be the torus $[k]^m$. That is the graph with vertex set $[k]^m$ in which vertices $(x_1,\ldots,x_m)$ and $(y_1,\ldots,y_m)$ are adjacent if $x_j=y_j \pm 1\mod m$ for exactly one coordinate $j$ while $x_i=y_i$ for all $i\not=j$. Let $\T^m$ be the sequence $(T^m_k)_{k=1}^{\infty}$

\begin{thm}\label{convex:thm}
The set $\Pcal(\T^m)$ is convex.
\end{thm}

\begin{proof}
Suppose that $(x_1,y_1), (x_2,y_2)\in \Pcal(\T^m)$ and $0<\lambda<1$. We will show that $\lambda(x_1,y_1)+(1-\lambda)(x_2,y_2)\in \Pcal(\T^m)$.

Given $\epsilon>0$, pick $k$ large enough that there is both a balanced colouring $C^1=(R^1,B^1)$ of $T^m_k$ with $|P_2(R^1)-x_1|,|P_2(B^1)-y_1|<\epsilon$ and a balanced colouring $C^2=(R^2,B^2)$ of $T^m_k$ with $|P_2(R^2)-x_2|,|P_2(B^2)-y_2|<\epsilon$. Also pick $s,t\in\N$ so that $|\frac{s}{t^m}-\lambda|\leq\epsilon$. 

We construct a balanced colouring $(R,B)$ of $[kt]^m$ by partitioning it into $t^m$ copies of $[k]^m$ in the obvious way and colouring $s$ of these copies with colouring $C^1$ and $t^m-s$ of them with colouring $C^2$. This clearly gives a balanced colouring. We say that a point of $[kt]^m$ is a bad point if it is close enough to the boundary of one of the partitioning copies of $[k]^m$ that it is is possible for a $2$-step walk from that point to reach a different partitioning $[k]^m$.The number of bad points is only $O(k^{m-1})$ so their contribution to $P_2(R)$ will only be $O(\frac{1}{k})$.

Suppose that $u\in[kt]^m$ is not bad. Then the probability that a $2$-step random walk starting from $u$ stays within its colour class can be found by conditioning on whether the copy of $[k]^m$ containing it was coloured with colouring $C^1$ or $C^2$. Indeed, for the balanced colouring $(R,B)$ of $[kt]^m$ we have
\[
P_2(R)=\frac{s}{t^m}P_2(R^1)+\left(1-\frac{s}{t^m}\right)P_2(R^2)+O\left(\frac{1}{k}\right).
\]
By the choice of $s$ and $t$ it follows that
\[
|P_2(R)-\lambda x_1-(1-\lambda)x_2|<c\epsilon+O\left(\frac{1}{k}\right)
\]
for some constant $c$. But the righthand side can be made arbitrarily small by choosing $\epsilon$ and $k$ suitably. 

Similarly
\[
|P_2(B)-\lambda y_1-(1-\lambda)y_2|<c\epsilon+O\left(\frac{1}{k}\right).
\]
It follows that $\lambda(x_1,y_1)+(1-\lambda)(x_2,y_2)\in \Pcal(\T^m)$ as required.
\end{proof}

A consequence of this result is that determining $\Pcal(T^m)$ reduces to finding its extreme points. Since $T^m_k$ is $2m$-regular we should be comparing this with $D_{2m}$.

The proof of Theorem \ref{main:thm} suggests that in order to attain the extreme points of $D_{2m}$ we need a colouring in which for all but $o(|V|)$ of the vertices, every vertex in $R$ has either $0$ or $2m$ neighbours in $R$ and every vertex in $B$ has the same number of neighbours in $B$ (or equivalently the same number of neighbours in $R$) or vice versa. The following definition, which we make for a general regular graph, formalises this idea and is key to constructing colourings which achieve extreme points. 

\begin{defn} 
Let $G$ be a $d$-regular graph and $0\leq r\leq d$. A set $S\subseteq V(G)$ is an \emph{independent exact $r$-cover} if it is an independent set in the graph $G$ and every element of $S^c$ is adjacent to exactly $r$ elements of $S$.
\end{defn}

By double-counting edges between $S$ and $S^c$ we see that the density (or asymptotic density if $G$ is infinite) of any independent exact $r$-cover is $\frac{r}{d+r}$.

Although we will use our constructions for the finite torus, it is convenient to express then in the infinite grid $\Z^m$, that is the graph with vertex set $\Z^m$ and two vertices adjacent if they differ by adding or subtracting $1$ in exactly one coordinate.

We first show how to use these objects to construct colourings of the torus with particular $P_2(R,B)$. The idea is to split the graph into two parts. In the first part we take red vertices forming an independent exact $r$-cover with the remaining vertices blue, in the second part we colour all vertices red. This will ensure that with a small number of exceptions, the red vertices will all have $0$ red neighbours or $m$ red neighbours (depending on whether they are in the first of second part) while the blue vertices (which are all in the first part) will all have $2m-r$ blue neighbours. The size and shape of the parts are chosen so that the colouring is balanced, and the number of exceptional vertices (those near the boundary of the parts) is small.

\begin{thm}\label{cover}
If an independent exact $r$-cover of $\Z^m$ exists then 
\[
\left(\left(\frac{2m-r}{2m}\right)^2,\frac{2m-r}{2m}\right)\in \Pcal(\T^m).
\]
\end{thm}

\begin{proof}
Let $S$ be an independent exact $r$-cover of $\Z^m$.

Let $a=\lceil\frac{2m+r}{4m}k\rceil$. We partition $[k]^m$ into two parts $X=[k]^{m-1}\times[a]$ and $Y=[k]^m\setminus X$. Now, let $R=((S\cap X)\cup Y)\triangle E$ and $B=[k]^m\setminus R=(X\setminus S)\triangle E$ where $E$ is a small set chosen to make $|R|=\frac{k^m}{2}$.

We can count the number of vertices in $X\setminus S$ by double counting the edges from a vertex in $S\cap X$ to a vertex in $X\setminus S$. For all but $O(k^{m-1})$ vertices this is determined by the exact $r$-cover property so:
\[
r|X\setminus S|=2m(|X|-|X\setminus S|)+O(k^{m-1}).
\]
Rewriting this we obtain
\begin{align*}
(2m+r)|X\setminus S|&=2mak^{m-1}+O(k^{m-1})\\
&=\frac{2m(2m+r)k^{m}}{4m}+O(k^{m-1}).
\end{align*}
We conclude that
\[
|X\setminus S|=\frac{k^m}{2}+O(k^{m-1}).
\]
So, we may take $|E|<ck^{m-1}$ for some constant $c$ depending on $m$ and $r$.

With the possible exception of vertices in or within distance $2$ of $E$ and vertices near the boundary of $X$ (that is with final coordinate in $\{1,2,a-1,a\}$) every vertex in $B$ has exactly $2m-r$ neighbours in $B$ and moreover each of those neighbours has exactly $2m-r$ neighbours in $B$. Hence we obtain
\[
P_2(B)=\left(\frac{2m-r}{2m}\right)^2+O\left(\frac{1}{k}\right).
\]

Similarly, proportion $2(1-\frac{a}{k})+o(1)=\frac{2m-r}{2m}+o(1)$ of the vertices in $R$ are at distance greater than $2$ from the nearest vertex in $B$ and 
proportion $\frac{r}{2m}+o(1)$ of the vertices in $R$ have all of their neighbours in $B$. Hence we obtain
\[
P_2(R)=\frac{2m-r}{2m}+o(1).
\]
By choosing $k$ large enough we get that 
\[
\left(\left(\frac{2m-r}{2m}\right)^2,\frac{2m-r}{2m}\right)\in \Pcal(\T^m)
\]
as required.
\end{proof}

We will give some constructions of independent exact $r$-covers in $\Z^m$ valid for several values of $r$. We also show that it is easy to scale up both $r$ and $m$ by a constant factor, so the most interesting open cases are those when $r$ and $m$ are coprime. 

\begin{thm}\label{cover-construction}
\begin{enumerate}
\item For any $m\geq 2$, an independent exact $r$-cover of $\Z^m$ exists when $r=0,1,2,m$ or $2m$.
\item If an independent exact $r$-cover of $\Z^m$ exists then so does an independent exact $\lambda r$-cover of $\Z^{\lambda m}$ for all $\lambda\in\N$.  
\end{enumerate}
\end{thm}

\begin{proof}
\begin{enumerate}
\item For $r=0$, the set $S=\emptyset$ is an independent exact $r$-cover.

For $r=1$, the set 
\[
S_1=\{(x_1,\ldots,x_m)\in\Z^m : \sum_{i=1}^m ix_i\equiv 0 \mod (2m+1)\}
\]
is an independent exact $1$-cover. Suppose that $x\in S_1$ and $y\in\Z^m$ differs from $x$ by changing coordinate $j$ by $\pm1$. Then $\sum_{i=1}^m ix_i-\sum_{i=1}^m iy_i=\pm j$ and so we cannot have $y\in S_1$. It follows that $S_1$ is an independent set. Also, if $y\in \Z^m\setminus S_1$ then $\sum_{i=1}^m iy_i\equiv r \mod (2m+1)$ with $r\not=0$. If $r\in\{1,2,\ldots, m\}$ then there is exactly one neighbour of $y$ in $S_1$ obtained by decreasing the $r$th coordinate of $y$ by $1$. If $r\in\{-1,-2,\ldots, -m\}$ then again there is exactly one neighbour of $y$ in $S_1$, this time obtained by increasing the $-r$th coordinate of $y$ by $1$.

For $r=2$, the set 
\[
S_2=\{(x_1,\ldots,x_m)\in\Z^m : \sum_{i=1}^m ix_i\equiv 0 \mod (m+1)\}
\]
is an independent exact $1$-cover. Suppose that $x\in S_2$ and $y\in\Z^m$ differs from $x$ by changing coordinate $j$ by $\pm1$. Then $\sum_{i=1}^m ix_i-\sum_{i=1}^m iy_i=\pm j$ and so we cannot have $y\in S_2$. It follows that  $S_2$ is an independent set. Also, if $y\in \Z^m\setminus S_2$ then $\sum_{i=1}^m iy_i\equiv r \mod (m+1)$ with $r\not=0$. Now there are exactly two neighbours of $y$ in $S_2$ obtained by decreasing the $r$th coordinate of $y$ by $1$ or increasing the $(m+1-r)$th coordinate of $y$ by $1$.

For $r=m$, the set 
\[
S_m=\{(x_1,\ldots,x_m)\in\Z^m : \sum_{i=1}^m x_i\equiv 0 \mod 3\}
\]
is an independent exact $m$-cover. It is easy to see that $S_m$ is an independent set. If $y\in \Z^m\setminus S_m$ then either $\sum_{i=1}^m y_i\equiv 1 \mod 3$ in which case the $m$ neighbours obtained by decreasing any coordinate by $1$ are all in $S_m$, or $\sum_{i=1}^m y_i\equiv -1 \mod 3$ in which case the $m$ neighbours obtained by increasing any coordinate by $1$ are all in $S_m$.

Finally, for $r=2m$, it is easy to see that the set
\[
S_m=\{(x_1,\ldots,x_m)\in\Z^m : \sum_{i=1}^m x_i\equiv 0 \mod 2\}
\]
is an independent exact $m$-cover.
\item Suppose that $S$ is an independent exact $r$-cover of $\Z^m$. Let $f:\Z^{\lambda m}\rightarrow\Z^m$ be defined by
\[
f(x_1,\ldots,x_{\lambda m})=(s_1,\ldots,s_m)
\]
where $s_k=\sum_{i=\lambda(k-1)+1}^{\lambda k}x_i$. In other words, we split the coordinates of $\Z^{\lambda m}$ into $m$ blocks of length $\lambda$ and let the $i$th coordinate of the image of $x$ be the sum of the coordinates in its $i$th block.

Now, define $S'\subseteq\Z^{\lambda m}$ by
\[
S'=f^{-1}(S)=\{x\in\Z^{\lambda m} : f(x)\in S\}.
\]
It is easy to see that $S'$ is an independent set. Now suppose that $y\in\Z^{\lambda m}\setminus S'$. Then $f(y)\in\Z^m\setminus S$ and so there are exactly $r$ neighbours of $f(y)$ in $S$. Each such neighbour gives rise to $\lambda$ neighbours of $y$ in $\Z^{\lambda m}$. For instance if increasing coordinate $i$ of $f(y)$ by $1$ yields an element of $S$ then increasing any of the coordinates $(i-1)\lambda+1,(i-1)\lambda+2,\ldots i\lambda$ of $y$ by $1$ yields an element of $S'$. Moreover each neighbour of $y$ in $S$ arises in this way and so $y$ has exactly $\lambda r$ neighbours in $S'$. It follows that $S'$ is an independent exact $\lambda r$-cover of $\Z^{\lambda m}$.
\end{enumerate}
\end{proof}

Using the constructions in Theorem \ref{cover-construction} and the connection between exact $r$-covers and $P_2(R,B)$ in Theorem \ref{cover} we obtain the following corollary.

\begin{cor}\label{construction}
Let $X_{2m}$ be the symmetric convex hull:
\[
\sch(\left\{\left(\frac{l}{2m},\left(\frac{l}{2m}\right)^2\right) : l=0,m,2m-2,2m-1,2m\right\}
\]
Then
\[
X_{2m}\subseteq\Pcal(\T^m)\subseteq D_{2m}
\]
\end{cor}

For $[k]^2$ we have an exact result.

\begin{thm}\label{2d-torus}
\[
\Pcal(\T^2)=\sch\left(\left\{\left(0,0\right),\left(\frac{1}{2},\frac{1}{4}\right),\left(\frac{3}{4},\frac{9}{16}\right),\left(1,1\right)\right\}\right)
\]
\end{thm}

Note that the difference between the extreme points of $D_4$ and the set $\Pcal(\T^2)$ in the statement of the Theorem is that the latter does not include $(\frac{1}{4},\frac{1}{16})$ or $(\frac{1}{16},\frac{1}{4})$. 

Recall the notation that $R_i$ is the set of all vertices in $R$ which have exactly $i$ neighbours in $R$ (and similarly for $B_i$). In order to attain $(\frac{1}{4},\frac{1}{16})$ we would need to find a colouring in which almost all vertices in $R$ lie in $R_0\cup R_4$ while almost all vertices in $B$ lie in $B_1$. We will show this is impossible by showing that $B_1$ being large forces many vertices in $R_1\cup R_2\cup R_3$.

\begin{proof}
By Theorem \ref{main:thm} we have
\[
\Pcal(\T^2)\subseteq\sch\left(\left\{\left(0,0\right),\left(\frac{1}{4},\frac{1}{16}\right),\left(\frac{1}{2},\frac{1}{4}\right),\left(\frac{3}{4},\frac{9}{16}\right),\left(1,1\right)\right\}\right).
\]
By the constructions of Corollary \ref{construction} when $m=2$, $r=1,2$ we have
\[
\sch\left(\left\{\left(0,0\right),\left(\frac{1}{2},\frac{1}{4}\right),\left(\frac{3}{4},\frac{9}{16}\right),\left(1,1\right)\right\}\right)\subseteq\Pcal(\T^2).
\]
So to prove the Theorem it suffices to show that if $(\frac{1}{4},y)\in\Pcal(\T^2)$ then $y\geq\frac{1}{8}$. 

Let $(R,B)$ be a balanced colouring of $[k]^2$. Suppose that $v\in B_1$. Let $v=(a,b)$ and suppose, without loss of generality, that the single neighbour of $v$ in $B$ is $w=(a,b+1)$. We will split $B_1$ into types as follows. We say that $v$ has:

\begin{itemize}
\item[Type 3:] if $w\in B_3\cup B_4$
\item[Type 1:] if $(a-1,b+1),(a+1,b+1)\in R$  
\item[Type 2:] otherwise
\end{itemize}
Note that if $w\in B_1$ then $v$ certainly has Type $1$, but if $w\in B_2$ then $v$ could be of Type 1 or Type 2. Let 
\[
t_i=|\{v\in B_1 : \text{$v$ has Type $i$}\}|
\]

We define an auxiliary bipartite graph $H$ with bipartition $(X,Y)$ where
\begin{align*}
X&=\{ v\in B_1 : \text{$v$ has Type $1$ or Type $2$}\}\\
Y&=R_1\cup R_2\cup R_3
\end{align*}

The edges of $H$ are of two types. If $v\in X$ and $r\in Y$ then $v$ and $r$ are adjacent in $H$ if:
\begin{itemize}
\item $v$ has Type $1$ and $v$ and $r$ are adjacent in $T_k^2$,
\item $v=(i,j)$ has Type $2$ and $r\in\{(i,j\pm1), (i\pm1,j),(i\pm1,j\pm1)\}$.
\end{itemize}

\begin{claim}
If $v\in X$ then $\deg_H(v)\geq 2$
\end{claim}

\begin{proof}[Proof of Claim 1.]
If $v$ has Type $1$ then we know $(a-1,b+1),(a+1,b+1)\in R$. Hence both $(a+1,b)$ and $(a-1,b)$ have at least one neighbour in $R$ and one neighbour in $B$ and so are elements of $Y$. It follows that $v$ is adjacent to both $(a+1,b)$ and $(a-1,b)$ in $H$ and so $\deg_H(v)\geq 2$.

If $v$ has Type $2$ then without loss of generality $(a-1,b+1)\in R, (a+1,b+1)\in B$ (if both these vertices were in $B$ then $v$ would have Type $3$). Hence both $(a-1,b+1)$ and $(a-1,b)$ are in $R$ and have at least one neighbour in $R$ and one neighbour in $B$, and so are elements of $Y$. It follows that $v$ is adjacent to both $(a-1,b+1)$ and $(a-1,b)$ in $H$ and so $\deg_H(v)\geq 2$.
\end{proof}

\begin{claim}
If $r\in Y$ then $\deg_H(r)\leq 3$
\end{claim}

\begin{proof}[Proof of Claim 2.]
Suppose that $r=(i,j)\in Y$. We need to show that among the $8$ points $\{(i,j\pm1), (i\pm1,j),(i\pm1,j\pm1)\}$ at most $3$ of them are elements of $X$ which are adjacent to $(i,j)$ in $H$. We will do this by first showing that we cannot have two such elements of $X$ which are adjacent (in $T_k^2$). We then rule out the two remaining $4$ point configurations.  

Suppose, for a contradiction, that both $(i+1,j)$ and $(i+1,j+1)$ are adjacent to $(i,j)$ in $H$. We must have that $(i+1,j)$ and $(i+1,j+1)$ are both in $B_1$ and so have no other neighbours in $B$. Hence both $(i+1,j)$ and $(i+1,j+1)$ are Type $1$ vertices. Now $(i+1,j+1)$ cannot be adjacent to $(i,j)$ in $H$ because of the definition of the edges of $H$ involving type 1 vertices. 

It follows that the only ways $(i,j)$ can have at least $4$ neighbours in $H$ are for these neighbours to be
\begin{itemize}
\item $(i\pm 1,j), (i,j\pm 1)$

or

\item $(i\pm 1,j\pm 1)$
\end{itemize}
In the first case we have $(i,j)\in R_0$. A contradiction since we must have $(i,j)\in Y=R_1\cup R_2\cup R_3$. In the second case, note that $(i+1,j)$ must be in $R$; if it were in $B$ then $(i+1,j+1)$ would have either Type 1 (if $(i+2,j)\in R$) or Type 3 (if $(i+2,j)\in B$) contradicting that $(i+1,j+1)$ is a neighbour of $(i,j)$ in H. Similarly, $(i-1,j),(i,j+1),(i,j-1)\in R$ and so $(i,j)\in R_4$. Again, this contradicts the definition of $Y$. This establishes the second claim.
\end{proof}

Using these two claims to double count the edges of $H$ we get:
\[
2(t_1+t_2)\leq E(H)\leq 3(r_1+r_2+r_3).
\]

Let $n=k^2$. Suppose that $P_2(R)=\frac{1}{4}+\eps$ then by Lemma \ref{squares}:
\[
\frac{1}{4}+\eps=\frac{1}{8n}\left(r_1+4r_2+9r_3+16r_4\right)
\]
and so
\[
2n+8\eps n=r_1+4r_2+9r_3+16r_4.
\]
We also have 
\[
I(R)=r_1+2r_2+3r_3+4r_4.
\]
Hence we obtain
\begin{equation}\label{pr}
2n+8\eps n=4I(R)-(3r_1+4r_2+3r_3)\leq 4I(R)-3(r_1+r_2+r_3)\leq 4I(R)-2(t_1+t_2)
\end{equation}

Now, consider $P_2(B)$. We have that
\[
8nP_2(B)=b_1+4b_2+9b_3+16b_4.
\]
Using the fact that $I(R)=I(B)=b_1+2b+2+3b_3+4b_4$ we can eliminate $b_2$ from this obtaining 
\begin{align*}
8nP_2(B)&=b_1+2\left(I(R)-b_1-3b_3-4b_4\right)+9b_3+16b_4\\
&=2I(R)-b_1+3b_3+8b_4.
\end{align*}
We know that $b_1=t_1+t_2+t_3$. We also know that $3b_3+8b_4\geq 3b_3+4b_4\geq t_3$ because each Type $3$ vertex in $B_3$ is adjacent to a vertex in $B_3$ or $B_4$. Using these bounds in the expression for $8nP_2(B)$ above gives
\begin{align*}
8nP_2(B)&\geq 2I(R)-(t_1+t_2)\\
&\geq n +4\eps n
\end{align*}
by equation (\ref{pr}). Finally, we obtain
\[
P_2(B)\geq\frac{1}{8}+\frac{\eps}{2}
\]
\end{proof}

Turning to larger $m$, we can give a similar but cruder non-existence argument for independent exact $r$-covers of $\Z^m$ when $m<r<2m$. 

\begin{thm}\label{no-exact-cover:thm}
For any $m\geq 2$. If $m<r<2m$ there is no independent exact $r$-cover of $\Z^m$.
\end{thm}

\begin{proof}
Suppose that $S\subseteq\Z^m$, $m<r<2m$ and that every point in $\Z^m\setminus S$ has exactly $r$ neighbours in $S$. Let $e_i\in\Z^m$ be the vector with $1$ in position $i$ and $0$s elsewhere. Take $x,y\in\Z^m\setminus S$ with $x=y+e_i$ (so $x,y$ are adjacent in $\Z^m$ and differ in coordinate $i$). We can certainly find such a pair since $r\not=2m$. 

Let $E=\{e_j: j\not=i\}\cup\{-e_j : j\not=i\}$. Let $X=\{e\in E :x+e\in S\}$ and $Y=\{e\in E :y+e\in S\}$. Since $x$ and $y$ each have $r$ neighbours in $S$ we have that $|X|\geq r-1$ and $|Y|\geq r-1$ (the bound is $r-1$ rather than $r$ because $x+e_i$ and $y-e_i$ may be in $S$). But $|E|=2m-2<2(r-1)$ so by the pigeonhole principle $X\cap Y\not=\emptyset$. If $e\in X\cap Y$ then $x+e,y+e$ are adjacent elements of $S$ and so $S$ cannot be an independent set.  
\end{proof}

This argument only shows the non-existence of a set exactly meeting the definition of an independent exact $r$-cover. However, since the argument is a very local one, the same idea shows that $\Pcal(\T^m)$ does not contain the points $\left(\frac{l}{2m},\frac{l^2}{(2m)^2}\right)$ with $1\leq l\leq m-1$. We sketch the argument below.
 
In order for the point $\left(\frac{l}{2m},\frac{l^2}{(2m)^2}\right)$ to be in $\Pcal(\T^m)$, we need colourings of $T_k^m$ for large $k$ with $P_2(R,B)$ arbitrarily close to $\left(\frac{l}{2m},\frac{l^2}{(2m)^2}\right)$ to exist. We say that a vertex in a colouring is good if it is either in $R$ and has $0$ or $2m$ neighbours in $R$, or it is in $B$ and has exactly $l$ neighbours in $B$. Looking at the proof of Theorem \ref{main:thm} we see that for the required colourings to exist we need all but $o(k^m)$ of the vertices to be good. 

In such a colouring, we can certainly find a set $Q$ of $ck^m$ pairs $x,y$ where $x$ and $y$ are adjacent good vertices in $B$ for some constant $c>0$. Now, when $1\leq l\leq m-1$ the argument of Theorem \ref{no-exact-cover:thm} shows that for each pair in $Q$ we can find an element of $R$ which is bad (i.e. with number of neighbours in $R$ strictly between $0$ and $2m$). Since each such element of $R$ can arise from a bounded number of pairs in $Q$ we have at least $c'k^m$ bad elements of $R$ for some constant $c'>0$. This contradiction completes the argument.

We have not attempted to make this more quantitative since the improvement would be small and probably far from the truth. We do not know the exact set $\Pcal(\T^m)$ for $m\geq 3$.

\section{Questions}

The most immediate open problem is to prove an analogue of Theorem \ref{2d-torus} for higher dimensions.

\begin{q}
What is $\Pcal(\T^m)$ for $m\geq 3$?
\end{q}

It could be that, as in the $m=2$ case, the answer is 
\[
\sch\left(\left\{(0,0)\right\}\cup\left\{\left(\frac{l}{2m},\frac{l^2}{(2m)^2}\right):m \leq l \leq 2m\right\}\right).
\]
For $m=3$, Theorem \ref{construction} shows that this set is certainly contained in $\Pcal(\T^3)$ but the true answer may have additional extreme points. For $m\geq 4$, we are missing both constructions and non-existence results.

Closely related to this is the question of when an independent exact $r$-cover of $\Z^m$ exists. 
\begin{q}
For which $r,m$ does an independent exact $r$-cover of $\Z^m$ exist?
\end{q}
As we have seen, these do exist for $r=0,1,2,m,2m$ and do not exist for $r=m+1,m+2,\ldots,2m-1$. All the missing cases of $3\leq r\leq m-1$ are open although we can reduce this to the cases when $r$ and $m$ are coprime (by the second part of Theorem \ref{cover-construction}). In particular, the first open case is $m=4, r=3$. 

We remark that all the constructions in the proof of Theorem \ref{cover-construction} are based on a divisibility condition on a linear combination of the coordinates in $\Z$. A different type of construction of an independent exact $2$-cover in $\Z^3$ is to take all points in $\Z^3$ for which the $3$ coordinates are either all even or all odd. Similarly, when $m=2^l-1$ for some $l\in\N$ we can construct an independent exact $2$-cover by taking all points in $\Z^m$ for which the coordinatewise reduction modulo $2$ gives an element of a fixed Hamming code (the $m=3$ case just described corresponds to the code $\{000,111\}$). 

As well as determining $\Pcal(\G)$ for a particular sequence of graphs, we can ask is there a sequence of graphs for which $\Pcal(\G)$ attains the maximum given by Theorem \ref{main:thm}.

\begin{q}
For all $d\in\N$ is there a sequence of $d$-regular graphs $\G$ for which $\Pcal(\G)=D_d$?
\end{q}

One way of constructing such a sequence would be to find a finite graph with the property that it contains an independent exact $r$-cover for all $0\leq r\leq d$. If we can do this then this graph has colourings with $P_2(R,B)$ achieving each of the extreme points of $D_d$. The sequence of $d$-regular graphs formed by taking an increasing number of disjoint copies of $G$ will have $\Pcal(\G)=D_d$.

A construction of such a graph for $d=3$ is as follows. Take
\[
V=\{a_i : 1\leq i\leq 20\}\cup\{b_i : 1\leq i\leq 20\}
\]
with the vertices $a_1,a_2,\ldots,a_{20}$ and $b_1,b_2,\ldots,b_{20}$ each forming a $20$-cycle. We will add to this a matching between $\{a_i : 1\leq i\leq 20\}$ and $\{b_i : 1\leq i\leq 20\}$ in such a way that in the resulting $3$-regular graph the set
\[
S_1=\{a_4,a_8,a_{12},a_{16},a_{20},b_4,b_8,b_{12},b_{16},b_{20}\}
\]
is an independent exact $1$-cover and the set
\[
S_2=\{a_1,a_3,a_6,a_8,a_{11},a_{13},a_{16},a_{18},b_1,b_3,b_6,b_8,b_{11},b_{13},b_{16},b_{18}\}
\]
is an independent exact $2$-cover. It suffices to take
\[
F=\{(1,5),(2,12),(3,9),(4,18),(6,20),(7,17),(8,14),(10,16),(11,15),(13,19)\}
\]
and  edges
\[
\{(a_i,b_j) : (i,j)\in F\}\cup \{(b_i,a_j) : (i,j)\in F\}.
\]
Since the graph we have constructed is bipartite it also has an independent exact $3$-cover given by one part of the bipartition.

It would be interesting to know if something like this can be done for $d\geq 4$.

\begin{q}
For all $d\in\N$ is there a finite $d$-regular graph $G$ which contains an independent exact $r$-cover for all $0\leq r\leq d$?
\end{q}

Note that the density condition on independent exact $r$-covers forces some divisibilty conditions on the possible size of $G$. The number of vertices $n$ in such a graph must be chosen so that $\frac{rn}{d+r}$ is an integer for all $0\leq r\leq d$.

\noindent
{\bf Note added in proof}: A positive answer to Question 13 has been given by Hou Tin Chau in the very recent preprint \cite{chau}.

Finally, a natural direction to consider is extending this work to $t$-step walks.

\begin{q}
What are the possible values of $P_t(R,B)$ for the torus $T^m_k$?
\end{q} 

The analogue of Theorem \ref{convex:thm} still holds for $t$-step walks (with the same proof) and it is again most natural to work asymptotically. Our construction based on independent exact $r$-covers can be analysed with respect to $t$-step walks. The construction asymptotically achieving the extreme point $\left(\frac{2m-r}{2m},\left(\frac{2m-r}{2m}\right)^2\right)$ for $P_2(R,B)$ also asymptotically achieves $\left(\frac{2m-r}{2m},\left(\frac{2m-r}{2m}\right)^t\right)$ for $P_t(R,B)$. However, we do not have an analogue of Theorem \ref{main:thm} for $t>2$. One obstacle to proving such a result is that the method of counting $2$-step walks using the squares of degrees does not generalise to $t$-step walks. 

\begin{q}
Is there a general bound analogous to Theorem \ref{main:thm} for the possible values of $P_t(R,B)$ in a $d$-regular graph?
\end{q} 

\noindent {\bf Acknowledgement.} We thank David Ellis for useful conversations and the two anonymous referees for their careful reading and comments which improved the exposition of this paper. 

\bibliography{Bibliography}
\bibliographystyle{abbrv}	

\end{document}